\documentclass[12pt,fleqn]{amsart}
\usepackage{amssymb,epsf,amsfonts,amsmath,amscd}
\usepackage{graphicx}
\usepackage[all]{xy}
\usepackage[T1]{fontenc}
\usepackage[a4paper]{geometry}
\usepackage{color}
\newtheorem{theorem}{Theorem}[section]
\newtheorem{corollary}[theorem]{Corollary}
\newtheorem{lemma}[theorem]{Lemma}
\newtheorem{proposition}[theorem]{Proposition}
\theoremstyle{definition}
\newtheorem{definition}[theorem]{Definition}
\newtheorem{remark}[theorem]{Remark}

\newtheorem{subsec}[theorem]{}
\newcommand{\End}{\operatorname{End}}
\newcommand{\Br}{\operatorname{Br}}
\newcommand{\Ker}{\operatorname{Ker}}
\newcommand{\Tr}{\operatorname{Tr}}
\newcommand{\Ind}{\operatorname{Ind}}
\newcommand{\Res}{\operatorname{Res}}
\newcommand{\Aut}{\operatorname{Aut}}
\newcommand{\Od}{\mathcal{O}}
\begin{document}

\title[]{Group graded basic Morita equivalences}
\author[Tiberiu Cocone\c t and Andrei Marcus]{Tiberiu Cocone\c t and Andrei Marcus}
\address{Babe\c s-Bolyai University, Faculty of Mathematics and Computer Science,  \newline Str. Mihail Kog\u alniceanu 1, RO-400084 Cluj-Napoca, Romania}
\email{tiberiu.coconet@math.ubbcluj.ro} \email{marcus@math.ubbcluj.ro}
\date{}

\abstract   We introduce group graded basic Morita  equivalences between algebras determined by blocks of normal subgroups, and by using the extended Brauer quotient, we show that they induce  graded basic Morita  equivalences at local levels. \endabstract


\maketitle
\section{Introduction}

Categorical equivalences between blocks of group algebras have been intensely studied over the last four decades, as they provide a structural explanation for various character correspondences which have been observed much earlier. It has turned out that in many cases, these are not mere equivalences between algebras -- they are also  compatible with the $p$-local structure of the blocks, encoded in terms of Brauer pairs or local pointed groups associated to subgroups of the defect groups, and their fusions.   The motivation is related to the older idea of obtaining information about the block from information about blocks of local subgroups.

In the case of principal blocks, J.~Rickard \cite{Ri} introduced the so-called splendid derived equivalences, which involve permutation bimodules with diagonal vertices and trivial sources, with the feature that they give rise to derived equivalences between principal blocks of centralizers of $p$-subgroups. Splendid equivalences have been generalized to arbitrary blocks by M.~Linckelmann \cite{L1}, \cite{L2} and M.E.~Harris \cite{Ha}. Later, a far reaching generalization has been achieved by L.~Puig \cite{Puig}, who introduced basic equivalences, which also have a significant local structure, and induce equivalences between block of the centralizers of subgroups of the defect groups by employing the Brauer construction. Moreover, L.~Puig and Y.~Zhou \cite{Puig3}, \cite{PZ2} proved that these local equivalences extend to equivalences between blocks of the normalizers of subgroups of the defect groups. It was pointed out by X.~Hu \cite{Hu} that in fact, one gets in \cite{Puig3} graded Morita equivalences, via the construction introduced in \cite{Marcus}.

In this paper, we investigate the local structure of graded Morita equivalences in a general setting. We fix a complete discrete valuation ring of characteristic zero with residue field $k$ of characteristic $p>0$,  finite groups $G$ and $G'$ with normal subgroups $N$ and $N'$ respectively, such that the factor groups $G/N$ and $G'/N'$ are isomorphic. Let $\ddot G$ be the diagonal subgroup of $G\times G'$ with respect to the given isomorphism between $G/N$ and $G'/N'$. Let $b$ be a $G$-invariant block of $\mathcal{O}N$, and let $b'$ be a $G'$-invariant block of $\mathcal{O}N'$.

A $G/N$-graded bimodule $\ddot M$ inducing a Morita equivalence between $A=\mathcal{O}Gb$ and $A'=\mathcal{O}G'b'$ has an $1$-component which can be regarded as an indecomposable $\mathcal{O}\ddot{G}$-module, so it has a vertex $\ddot P$ and a source $\ddot N$. The discussion of the Morita equivalences in terms of $\ddot P$ and $\ddot N$  requires the introduction of a natural group graded structure on Puig's $\mathcal{O}G$-interior Hecke algebra $\End_{\mathcal{O}(1\times G')}(\Ind_{\ddot{P}}^{G\times G'}(\ddot{N}))$. We studied this graded structure in detail in \cite{CM}. In Section 3 we characterize in these terms the $G/N$-graded Morita equivalences between $\mathcal{O}Gb$ and $\mathcal{O}G'b'$, by linking, via the source module $\ddot N$, a defect pointed group $P_\gamma$ of $G_{\{b\}}$ on $\mathcal{O}Nb$ and a defect pointed group $P'_{\gamma'}$ of ${G'}_{\{b'\}}$ on $\mathcal{O}N'b'$. Our first main result, Theorem \ref{main} below, relates the Morita equivalence between the block extensions $A$ and $A'$ induced by $\ddot M$ to a $G/N$-graded Morita equivalence between the source algebras $A_\gamma$ and $A'_{\gamma'}$.

In Section 4 we introduce the graded version of  basic Morita equivalence between $\mathcal{O}Gb$ and $\mathcal{O}G'b'$, and we show in Corollary \ref{c:truncation} that the truncation (that is, restriction of the grading to a subgroup of $\Gamma$) of a group graded basic Morita equivalence is again a group graded basic Morita equivalence.

The extended Brauer quotient of $\mathcal{O}G$ with respect to a $p$-subgroup $Q$ of $G$ was introduced in \cite[Section 3]{Puig3}, and it was generalized to $N$-interior $G$-algebras in \cite{CT}. In our situation, we show in Section 5 that the $G/N$-grading on $\mathcal{O}G$ induces a group grading on the extended Brauer quotient, given by a certain subgroup of the normalizer $N_G(Q)$ modulo the centralizer $C_N(Q)$.

This observation is used in our second  main result, Theorem \ref{t:localequiv}  below, where we prove that a basic graded  Morita equivalence between $\mathcal{O}Gb$ and $\mathcal{O}G'b'$ induces a group graded basic equivalence at local levels. This is a common generalization of \cite[Theorem 1.4]{Puig3} and \cite[Corollary 3.9]{Marcus}. More precisely, \cite[Theorem 1.4]{Puig3} is obtained in the case $G=N$, while in \cite[Corollary 3.9]{Marcus} we have only principal blocks $b$ and $b'$, with Sylow $p$-subgroups $P\le N$ and $P'\le N'$, and we do not consider the extended Brauer quotient.

One should mention that Puig also discussed in his book \cite{Puig} the other types of equivalences relevant to block theory -- stable equivalences of Morita type and Rickard equivalences, and moreover, the results of \cite{Puig3} are extended in \cite{PZ2} to obtain local Rickard equivalences with a group graded structure. However, the comparison of the results on group graded equivalences presented in \cite[Theorem 5.1.2]{M} (Morita),  \cite[Theorem 5.2.5]{M} (Rickard) and \cite[Proposition 5.3.7]{M} (stable) suggests that the assumption that the characteristic $p$ does not divide the order of the grading group $G/N$ is quite reasonable in the case of stable and Rickard equivalences.

We introduce our notations and basic assumptions in Section 2. We will freely use the language of pointed groups on $G$-algebras and their fusions, as presented in \cite{Puig}, \cite{Puig1} and \cite{The}. We will not need to assume that $k$ is a splitting field for the group algebras in discussion, so certain field extensions of $k$ will occur below. We also refer to \cite{M} for results on group graded algebras and modules.

\section{Preliminaries} \label{sect:prelim}

The following notations and assumptions will be in force for the remaining of the paper.

\begin{subsec} \label{s:GN} Let $G$ and $G'$ be finite groups, let $\omega:G\to \Gamma$ and $\omega':G'\to \Gamma$ be group epimorphisms, and denote  $N:=\mathrm{Ker} \omega$ and $N':=\mathrm{Ker} \omega'$,  so that \[\Gamma\simeq G/N\simeq G'/N'.\] Let $\ddot G$ be the inverse image in $G\times G'$ of the diagonal subgroup of $\Gamma\times \Gamma$, that is
\[\ddot G=\{(g,g')\in G\times G'\mid \omega(g)=\omega'(g') \}.\] Consider the diagonal subalgebra
\[\Delta=\bigoplus_{(g,g')\in [\ddot G/N\times N']}\mathcal{O}Ng\otimes_{\mathcal{O}}\mathcal{O}N'g',\]
of the $\Gamma\times\Gamma$-graded algebra $\mathcal{O}G\otimes_{\mathcal{O}}\mathcal{O}G'$ (with respect to the maps $\omega$ and $\omega'$), so $\Delta$ is naturally isomorphic, as a $\Gamma$-graded algebra, to the group algebra $\mathcal{O}\ddot G$.
\end{subsec}

\begin{subsec} \label{assumption_on_M}  Let $b$ be a block of $\mathcal{O}N$ and let  $b'$ be a block of $\mathcal{O}N'$ such that $b$ is $G$-invariant and $b'$ is $G'$-invariant. Set $A:=\mathcal{O}Gb$ and $A':=\mathcal{O}G'b'.$ Then both $A$ and $A'$ are $\Gamma$-graded in an obvious way, and we denote by $A_1$ and $A_1'$ respectively the identity components.

We consider the $\Gamma$-graded $G$-acted $\mathcal{O}$-algebra structure on  ${A},$  that is, we have $a^g\in A_{x^g},$ for all $x\in \Gamma$ and $a_x\in A_x$.
\end{subsec}

\begin{subsec} \label{s:ddotM}  Let $M$ be an indecomposable $A_1\otimes A'_1$-module, so $M$ is an indecomposable $\mathcal{O}(N\times N')$-module associated with $b\otimes (b')^\circ$. We assume  that $M$ extends to $\Delta$ (that is, the action of $N\times N$ on $M$ extends to $\ddot{G})$, and that $M$ restricted to both $\mathcal{O}(1\times N')$ and $\mathcal{O}(N\times 1)$ is projective. Set
\[\ddot{M}:=\Ind_{\ddot G}^{G\times G'}(M).\] Then, by \cite[Lemma 1.6.3]{M}, $\ddot{M}$ is a $\Gamma$-graded $(\mathcal{O}G,(\mathcal{O}G')^{\mathrm{op}})$-bimodule with $1$-component naturally isomorphic to $M$.


Since $M$ is an indecomposable $\mathcal{O}\ddot G$-module,  we consider a vertex $\ddot{P}$ of $M$  in $\ddot G,$ and we choose an $\mathcal{O}\ddot{P}$-source $\ddot{N}$ of $M.$ Consider the $\mathcal{O}\ddot{P}$-interior algebra \[S:=\End_{\mathcal{O}}(\ddot{N}).\]

It follows in particular that $M$ is a direct summand of  $\Ind_{\ddot{P}}^{\ddot G}(\ddot{N})$ as $\mathcal{O}\ddot G$-modules,  and  $\ddot{M}$ is a direct summand of $\Ind_{\ddot{P}}^{G\times G'}(\ddot{N})$ as $\Gamma$-graded $\mathcal{O}(G\times G')$-modules, where note that the map $\ddot{P}\to G\times G'$ is not injective in general. Let $X$ be the Green correspondent of $M$, so $X$ is an indecomposable $\mathcal{O}N_{\ddot G}(\ddot{P})$-module with vertex $\ddot{P}$ such that $X\mid \Ind_{\ddot{P}}^{N_{\ddot G}(\ddot{P})}(\ddot{N}).$

Denote also by $\pi:G\times G'\to G$ the natural projection, by $\rho:\ddot{P}\to G,$ the restriction of $\pi$ to $\ddot{P}$ and by $\sigma:\ddot{P}\to P$ the surjective group homomorphism determined by $\pi$ and $\rho.$ Similarly, we set $\pi',$ $\rho'$ and $\sigma':\ddot{P}\to P'.$
\end{subsec}

For later use, we record the following lemma relating points on $A$ and points on $A_1$.

\begin{lemma} \label{point_relation} For any subgroup $H$ of $G$ and for any  point $\hat{\alpha}$ of $H$ on ${A}$ there exists a point $\alpha$  of $H$ on $A_1$ such that $\hat{i}\cdot j=j\cdot \hat{i}=\hat{i},$   for some $\hat{i}\in \hat{\alpha}$ and $j\in \alpha.$
\end{lemma}

\begin{proof} Let $1_{{A}}=\sum j$ be an orthogonal primitive decomposition of the identity element of ${A}$ in $(A_1)^H.$ Let $\hat{i}'\in \hat{\alpha}.$ Since $\hat{i}'\cdot 1_{{A}}=1_{{A}}\cdot \hat{i}'=\hat{i}'$ it follows that there exists an idempotent $j$ such that $j\cdot \hat{i}'\neq 0.$   Now \cite[Proposition 3.19]{Puig1} shows that there exists $\hat{a}\in {A}^*$ such that $j^{\hat{a}}\cdot \hat{i}'=\hat{i}'\cdot j^{\hat{a}}=\hat{i}'$ in ${A}^H,$ as $\hat{i}'$ is a primitive idempotent. The statement follows by taking $\hat{i}:=(\hat{i}')^{\hat{a}^{-1}}\in \hat{\alpha}$ and $\alpha $ the point containing $j.$
\end{proof}

\section{Graded Hecke interior algebras and Morita equivalences}

\begin{subsec} Let $\ddot{M}$ be an indecomposable $\Gamma$-graded $\mathcal{O}(G\times G')$-module with $1$-component $M$ having vertex $\ddot{P}$ and source $\ddot{N}$   as in \ref{s:ddotM}.  Consider the induced algebra
\[\ddot{A}:=\Ind_{\ddot{P}}^{G\times G'}(\End_{\mathcal{O}}(\ddot{N}))\simeq \End_{\mathcal{O}}(\Ind_{\ddot{P}}^{G\times G'}(\ddot{N})),\]
and set
\[\hat{A}:=\ddot{A}^{1\times G'}\simeq \End_{\mathcal{O}(1\times G')}(\Ind_{\ddot{P}}^{G\times G'}(\ddot{N})).\]
The algebra $\hat A$ is called in \cite[Section 4]{Puig} the {\it Hecke $\mathcal{O}G$-interior algebra} associated with $G'$, $\ddot{P}$ and $\End_{\mathcal{O}}(\ddot{N})$. In our situation,
by \cite[Section 6]{CM}, $\ddot{A}$ and $\hat{A}$ are naturally $\Gamma$-graded  $\mathcal{O}(G\times G')$-interior respectively $\mathcal{O}G$-interior algebras, such that the structural map
\[\mathcal{O}G\to \hat{A}\] is a homomorphism of $\Gamma$-graded $\mathcal{O}G$-interior algebras.
\end{subsec}

\begin{subsec} \label{s:isoms} By \cite[Theorem 4.4]{Puig} there is an $\mathcal{O}G$-interior algebra isomorphism
\begin{equation} \label{1}
\hat{A}\simeq \Ind_{\rho}(\End_{\mathcal{O}}(\ddot{N})\otimes_{\mathcal{O}}\Res_{\rho'}(\mathcal{O}G')),
\end{equation}
which in our case, by \cite[Section 7]{CM}, is an isomorphism of $\Gamma$-graded algebras. As a consequence, we obtain the isomorphism
\begin{equation}
\label{2}\Ind_{\ddot{P}}^{P\times G'}(\End_{\mathcal{O}}(\ddot{N}))^{1\times G'}\simeq \Ind_{\sigma}(\End_{\mathcal{O}}(\ddot{N})\otimes_{\mathcal{O}}\Res_{\rho'}(\mathcal{O}G'))
\end{equation}
of $\Gamma$-graded $\mathcal{O}P$-interior algebras. Note that constructions made in \cite[6.7]{Puig} can be applied in our situation since $b'$ is $G'$-invariant, hence fron (\ref{1}) we obtain the isomorphism
\begin{equation}
\label{3} \hat{A}\cdot (b')^\circ\simeq \Ind_{\rho}(\End_{\mathcal{O}}(\ddot{N})\otimes_{\mathcal{O}}\Res_{\rho'}(A'))\end{equation}
of $\Gamma$-graded $\mathcal{O}G$-interior algebras.
\end{subsec}

\begin{subsec} \label{n:endom-alg}   Denote $\bar N_{G\times G'}(\ddot P)=N_{G\times G'}(\ddot P)/\ddot P$ and $\bar N_{\ddot G}(\ddot P)=N_{\ddot G}(\ddot P)/\ddot P$. Consider the $\bar N_{\ddot G}(\ddot P)$-graded endomorphism algebra
\[E:=\End_{\mathcal{O}N_{\ddot G}(\ddot{P})}(\Ind_{\ddot{P}}^{N_{\ddot G}(\ddot{P})}(\ddot{N}))^{\mathrm{op}}\]
and the $\bar N_{G\times G'}(\ddot P)$-graded endomorphism algebra
\[\ddot{E}:=\End_{\mathcal{O}N_{G\times G'}(\ddot{P})}(\Ind_{\ddot{P}}^{N_{G\times G'}(\ddot{P})}(\ddot{N}))^{\mathrm{op}}.\]
The quotient algebra $\ddot E/\mathrm{J}_{\mathrm{gr}}(\ddot E)$ modulo the graded Jacobson radical $\mathrm{J}_{\mathrm{gr}}(\ddot E)$ is a $\bar N_{G\times G'}(\ddot{P}_{\ddot{N}})$-graded crossed product denoted $\bar k_*\hat{\bar{N}}_{G\times G'}(\ddot{P}_{\ddot{N}})$, where
\[\bar k=\End_{\mathcal{O}{\ddot P}}({\ddot N})/\mathrm{J}(\End_{\mathcal{O}{\ddot P}}({\ddot N}))\] is an extension of $k$, while $E/\mathrm{J}_{\mathrm{gr}}(E)$ is naturally isomorphic to the $\bar N_{G\times G'}(\ddot{P}_{\ddot{N}})$-graded subalgebra $\bar k_*\hat{\bar{N}}_{\ddot G}(\ddot{P}_{\ddot{N}})$
\end{subsec}

\begin{subsec} \label{n:Clifford-corr} By the Clifford correspondence (see, for instance, \cite[Section 2.3.B]{M}), the indecomposable $\mathcal{O}N_{\ddot G}(\ddot{P})$-summand $X$ of $\Ind_{\ddot{P}}^{N_{\ddot G}(\ddot{P})}(\ddot{N})$  determines an indecomposable projective $E$-module, which, in turn, determines the indecomposable projective $\bar k_*\hat{\bar{N}}_{\ddot G}(\ddot{P}_{\ddot{N}})$-module $V$, similarly to \cite[6.6]{Puig}. Since by \cite[Theorem 2.3.10.d)]{M} the Clifford correspondence commutes with induction, we get that $\ddot{M}$ corresponds to
\[\Ind_{N_{\ddot G}(\ddot{P})}^{N_{G\times G'}(\ddot{P})}(X),\]
which, again by Clifford correspondence, determines a projective $\bar k_*\hat{\bar{N}}_{G\times G'}(\ddot{P}_{\ddot{N}})$-module $\ddot{V}$ such that
\[\ddot{V}\simeq \Ind_{\bar k_*{\bar{N}}_{\ddot{G}}(\ddot{P}_{\ddot{N}})}^{\bar{k}_*{\bar{N}}_{G\times G'}(\ddot{P}_{\ddot{N}})}(V).\]
\end{subsec}

\begin{subsec}  \label{gammahat} Now let $\tilde{\ddot{V}}$ be an indecomposable direct summand of $\ddot{V}$ and let $\ddot{W}$ be an isotypic component of the restriction to $\bar k_*\hat{\bar{N}}_{P\times G'}(\ddot{P}_{\ddot{N}})$ of $\mathrm{\tilde{\ddot{V}}}.$ By \cite[6.8]{Puig}, $\ddot{W}$ determines a local point $\ddot{\hat{\gamma}}$ of $P$ on the induced algebra \[\Ind_{\sigma}(\End_{\mathcal{O}}(\ddot{N})\otimes_{\mathcal{O}}\Res_{\rho'}(\mathcal{O}G')),\]
occurring in \ref{s:isoms} (\ref{2}), and a local point $\hat{\gamma}'$ of $P'$ on $\mathcal{O}G'$, such that $\ddot{\hat{\gamma}}$ is a point of $P$ on \[\Ind_{\sigma}(\End_{\mathcal{O}}(\ddot{N})\otimes_{\mathcal{O}}\Res_{\rho'}(\mathcal{O}G')_{\hat{\gamma}'}).\]
\end{subsec}

Let $\tilde{\ddot{M}}$ be the indecomposable $\mathcal{O}(G\times G')$-summand of $\ddot{M}$ that corresponds to $\tilde{\ddot{V}}.$ Since $\ddot{M}$ is associated with $b\otimes (b')^o$, so is $\tilde{\ddot{M}},$ and then we obtain the $\mathcal{O}(G\times G')$-interior algebra embeddings
\[\End_{\mathcal{O}}(\tilde{\ddot{M}})\to \End_{\mathcal{O}}(\ddot{M})\to \ddot{A},\] which induce the $\mathcal{O}G$-interior algebra embeddings
\[\End_{\mathcal{O}}(\tilde{\ddot{M}})^{1\times G'}\to \End_{\mathcal{O}}(\ddot{M})^{1\times G'}\to \hat{A}.\]
Now $\tilde{\ddot{M}}$ determines, as in \cite[6.6]{Puig}, a point $\hat{\alpha}$ of  $G$ on $\End_{\mathcal{O}}(\ddot{M})^{1\times G'}$ and on $\hat{A}$ such that $b\cdot \hat{\alpha}=\hat{\alpha}.$ As in \cite[6.8]{Puig}, we obtain that
\[(P\times G')_{\ddot{\hat{\gamma}}}\leq (G\times G')_{\hat{\alpha}},\]
and thus, by using the isomorphism \ref{s:isoms} (\ref{3}), we have that $P'_{\hat{\gamma}'}$ is a local point on $A'.$

\begin{subsec}  Consider again the $\Gamma$-graded $\mathcal{O}G$-interior algebra
\[\hat{C}:=\Ind_{\rho}(\End_{\mathcal{O}}(\ddot{N})\otimes_{\mathcal{O}}\Res_{\rho'}(A'))\]
occurring in \ref{s:isoms} (\ref{3}), whose $1$-component is
\[\hat{C}_1=(\mathcal{O}\otimes_{\mathcal{O}(1\times N')} (\End_{\mathcal{O}}(\Ind_{\ddot{P}}^{\ddot G}(\ddot{N}))\otimes_{\mathcal{O}}A_1'))^{1\times N'}.\]
\end{subsec}

\begin{proposition} \label{points_on_1_cmp} With the above notation, the points $\ddot{\hat{\gamma}}$ and $\hat{\gamma}'$ introduced in \ref{gammahat} determine the local points $\hat{\gamma}$ of $P$ on $\hat{C}_1$ and $\gamma'$ of $P$ in $A'_1$ such that $\hat{\gamma}$ is a point of $P$ on \[(\hat{C}_1)_{\gamma'}:=(\mathcal{O}\otimes_{\mathcal{O}(1\times N')} (\End_{\mathcal{O}}(\Ind_{\ddot{P}}^{\ddot G}(\ddot{N}))\otimes_{\mathcal{O}}(A_1')_{\gamma'}))^{1\times N'}.\]
\end{proposition}

\begin{proof} It follows from \cite[6.8]{Puig} that $\ddot{\hat{\gamma}}$ is a point of $P$ on
\[\Ind_{\sigma}(\End_{\mathcal{O}}(\ddot{N})\otimes_{\mathcal{O}}\Res_{\sigma'}(A'_{\hat{\gamma}'})).\]
Since $A'$ is $\Gamma$-graded, Lemma \ref{point_relation} gives a point $\gamma'$ of $P'$ on $A'_1$  such that $\hat{i}'\cdot i'=\hat{i}'$ for some idempotents chosen as above. Now clearly $\ddot{\hat{\gamma}}$ becomes a point of $P$ on
\[\Ind_{\rho}(\End_{\mathcal{O}}(\ddot{N})\otimes_{\mathcal{O}}\Res_{\sigma'}(A'_{\gamma'}))\]
via the $\mathcal{O}P$-interior algebra embedding
\[\Ind_{\sigma}(\End_{\mathcal{O}}(\ddot{N})\otimes_{\mathcal{O}}\Res_{\sigma'}(A'_{\gamma'}))\to \Ind_{\rho}(\End_{\mathcal{O}}(\ddot{N})\otimes_{\mathcal{O}}\Res_{\sigma'}(A'_{\gamma'})).\]
This last algebra is again $\Gamma$-graded, having as identity component the $P$-algebra
$(\hat{C}_1)_{\gamma'}.$
The point $\hat{\gamma}$ is determined by applying again Lemma \ref{point_relation}.
\end{proof}

\begin{subsec} Recall that a pointed group $P_\gamma$ on $A_1$ determines, as in \cite[1.16]{FP}, the simple quotient $A_1(P_\gamma)$ of $A_1^P$, and  a field extension $\hat k$ of $k$ and a crossed product (or a twisted $\hat k^*$-group algebra) $\hat k_*\hat{\bar N}_G(P_\gamma)$, such that the simple $A_1(P_\gamma)$-module becomes a $\hat k_*\hat{\bar N}_G(P_\gamma)$-module denoted by $V_{A_1}(P_\gamma)$, and called the {\it multiplicity module} of $P_\gamma$. 

When we apply this idea to the $\mathcal{O}\ddot{G}$-module $M$ and the $\mathcal{O}\ddot{G}$-interior algebra $\End_{\mathcal{O}}(M)$, then the pointed group corresponding to the source $\ddot{N}$ will be denoted by $\ddot{P}_{\ddot{N}}$, and the multiplicity $\bar k_* \hat{{\bar N}}_{\ddot{G}}(\ddot{P})$-module will be denoted by $V_{M}(\ddot{P}_{\ddot{N}})$.

With the above notations, we may state the following group graded version of \cite[Theorem 6.9]{Puig}. Note again that we only consider here Morita equivalences (and not Morita stable equivalences), so our result takes the simpler form mentioned at the end of the statement of \cite[Theorem 6.9]{Puig}.
\end{subsec}

\begin{theorem} \label{main}  The $(\mathcal{O}G,\mathcal{O}G')$-bimodule $\ddot{M}$ induces a $\Gamma$-graded Morita equivalence between $A$ and $A'$  if and only if   $P$ is a defect group of $b$ regarded as a primitive idempotent of  $(A_1)^G,$  $P'$ is a defect group of $b'$ regarded as a primitive idempotent of  $(A_1')^{G'},$ and for a suitable local point $\gamma$ of $P$ on $A_1$, there is a $\Gamma$-graded $\mathcal{O}P$-interior algebra isomorphism
\[e:A_{\gamma}\longrightarrow (\Ind_{\sigma}(\End_{\mathcal{O}}(\ddot{N})\otimes_{\mathcal{O}}\Res_{\sigma'}(A'_{\gamma'})))_{\hat{\gamma}}\]
{\rm(}where  $\hat\gamma$ and $\gamma'$ are defined in Proposition \ref{points_on_1_cmp}{\rm)} such that we have an  isomorphism
\[V_{A_1}(P_{\gamma})\simeq \Res_{\widehat{(\Ind_P^G e)_1(P_{\gamma})}}(V_{M}((P\times G'))_{\hat{\gamma}}))\]
of $\hat k_*\hat{\bar{N}}_{G}(P_{\gamma})$-modules, where
\[\widehat{(\Ind_P^G e)_1(P_{\gamma})}: \hat k_*\hat{\bar N}_G(P_\gamma) \to \hat k_*\hat{\bar N}_G(P_{\hat\gamma})\simeq \hat k_*\hat{\bar N}_{G\times G'}((P\times G')_{\hat\gamma})\]
is the isomorphism of twisted group algebras induced by $e$.
\end{theorem}

\begin{proof}  By our assumptions,  the $\Gamma$-graded Morita equivalence determines the $\Gamma$-graded $\mathcal{O}G$-interior algebra isomorphism
\begin{equation}
\label{i}A\simeq \End_{\mathcal{O}}(\ddot{M})^{1\times G'},
\end{equation}
given by the structural map. We have that $N\times 1$, $1\times N'$ and $N\times N'$ are normal subgroups of ${\ddot G},$ so  the identity component of the above isomorphism restricts to the $\ddot G$-algebra isomorphism
\[A_1\simeq \End_{\mathcal{O}}(M)^{1\times N'},\] given by the structural map, as in \cite[Theorem 6.5]{Puig}. Therefore, we have
\[A_1^{\ddot G}\simeq \End_{\mathcal{O}}(M)^{\ddot G}.\] Note that the action of $\ddot G$ on $A_1$ coincides with the action of $G$ on $A_1,$ forcing $(A_1)^G\simeq (A_1)^{\ddot G}.$

Let $D$ be a defect group in $G$ of $b.$ If $D'$ is the subgroup of $G'$ such that $\omega(D)=\omega'(D')$ (so $\omega$ and $\omega'$ induce the isomorphism $DN/N\simeq D'N'/N'$), then
\[(A_1)^G_D\simeq (A_1)^{\ddot G}_{D\times D'}\simeq \End_{\mathcal{O}}(M)^{\ddot G}_{D\times D'}.\]
We deduce that for a suitable pair $(x,x')\in {\ddot G}$, we have $\ddot P^{(x,x')}\leq D\times D',$ that is, $P^x\leq D.$

Now let $\bar{\ddot{M}}$ be an indecomposable $\mathcal{O}(G\times G')$-summand of $\ddot{M}$, having vertex $\ddot{Q}$ and source $\bar{\ddot{N}}$, and let  $Q$ and $Q'$ the projections of $\ddot{Q}$ in $G$ and $G'$ respectively. Then $\bar{\ddot{M}}$ determines a block $B$ of $A$ such that by isomorphism (\ref{i}) we obtain the isomorphism
\[A\cdot B\simeq \End_{\mathcal{O}}(\bar{\ddot{M}})^{1\times G'}\]
of $\mathcal{O}G$-interior algebras.  Moreover, the block $B$ determines a point $\tilde{\alpha}$ of $G$ on
\[\Ind_{\ddot{Q}}^{G\times G'}(\End_{\mathcal{O}}(\bar{\ddot{N}}))^{1\times G'}\] such that the $\mathcal{O}G$-interior algebra embedding
\[ \End_{\mathcal{O}}(\bar{\ddot{M}})^{1\times G'}\to \Ind_{\ddot{Q}}^{G\times G'}(\End_{\mathcal{O}}(\bar{\ddot{N}}))^{1\times G'}\] determines the $\mathcal{O}G$-interior algebra isomorphism
\[A\cdot B\simeq (\Ind_{\ddot{Q}}^{G\times G'}(\End_{\mathcal{O}}(\bar{\ddot{N}}))^{1\times G'})_{\tilde{\alpha}}.\] Using this last isomorphism and \cite[Proposition 5.3]{Puig}, we deduce that $Q$ is actually a defect group of $B.$ On the other hand, we have $\bar{\ddot{M}}\mid \ddot{M}\mid \Ind_{\ddot{P}}^{G\times G'}(\ddot{N})$, and thus we may choose $Q$ to be a subgroup of $P.$ So far we have obtained the inclusions $Q^x\leq P^x\leq D$ for some $x\in G.$

By changing the choice of $\bar{\ddot{M}}$, we get, as above, all the blocks $B$ of $G$ that cover $b.$ It is known (see \cite[IV.15, Theorem 1]{Alp}) that there is at least one block that covers $b$ and has defect group $D$. If that  is the case for $B$, then we deduce that $P^x=D$ for some $x\in G.$ Hence $P$ is a defect group of $b$ in $G.$ By symmetry, we get that $P'$ is a defect group of $b'$ in $G'.$

Let $\bar{B}$ be the block of $A$  corresponding to $\hat{\alpha}$ via (\ref{i}). Then we have the isomorphisms
\[\bar{B}A\simeq (\End_{\mathcal{O}}(\bar{\ddot{M}})^{1\times G'})_{\hat{\alpha}}\simeq \hat{A}_{\hat{\alpha}},\]
of $\mathcal{O}G$-interior algebras. Clearly, the defect pointed group $P_{\ddot{\hat{\gamma}}}$ of $G_{\hat{\alpha}}$ determines a defect pointed group $P_{\bar{\gamma}}$  of $G_{\bar{B}}$, and $\bar{\gamma}$ is still a local point of $P$ on $A$ via the embedding $\bar{B}A\to A.$ Again Lemma \ref{point_relation} gives a point $\gamma$ of $P$ on $A_1,$ which is also local since $\bar{\gamma}$ is. It is now easy to see that $P_{\gamma}$ is a defect pointed group of $G_{\{b\}}.$

The $\Gamma$-graded $\mathcal{O}G$-interior algebra embedding
\begin{equation}
\label{4}A\simeq \End_{\mathcal{O}}(\ddot{M})^{1\times G'}\to \hat{A}\cdot (b')^{o}\simeq \Ind_{\rho}(\End_{\mathcal{O}}(\ddot{N})\otimes_{\mathcal{O}}\Res_{\rho'}(A'))\end{equation}
restricts to a $G$-algebra embedding
$A_1\to \hat{C}_1,$ so that the correspondence between $\bar{\gamma}$ and $\ddot{\hat{\gamma}}$ determines the correspondence between $\gamma$ and $\hat{\gamma}.$ Proposition \ref{points_on_1_cmp} gives the $\mathcal{O}P$-interior algebra isomorphism
\begin{equation}
\label{5}(A_1)_{\gamma}\simeq ((\hat{C}_1)_{\gamma'})_{\hat{\gamma}}\simeq (\Ind_{\sigma}(\End_{\mathcal{O}}(\ddot{N})\otimes_{\mathcal{O}}\Res_{\sigma'}((A'_1)_{\gamma'})))_{\hat{\gamma}}.\end{equation}
Recall that the embedding of $\mathcal{O}P$-interior algebras

\[(\Ind_{\ddot{P}}^{P\times G'}(\End_{\mathcal{O}}(\ddot{N}))^{1\times G'})\cdot (b')^o\to \hat{A}\cdot (b')^o\] is grade-preserving, while
\[\Ind_{\ddot{P}}^{P\times G'}(\End_{\mathcal{O}}(\ddot{N}))^{1\times G'}\cdot (b')^{o}\simeq \Ind_{\sigma}(\End_{\mathcal{O}}(\ddot{N})\otimes_{\mathcal{O}}(\Res_{\rho'}(A')))\]
is an isomorphism of $\Gamma$-graded $\mathcal{O}P$-interior algebras, and
\[\Ind_{\ddot{P}}^{G\times G'}(\End_{\mathcal{O}}(\ddot{N}))^{1\times G'}\cdot (b')^o\simeq \Ind_{\rho}(\End_{\mathcal{O}}(\ddot{N})\otimes_{\mathcal{O}}(\Res_{\rho'}(A')))\]
is an isomorphism of $\Gamma$-graded $\mathcal{O}G$-interior algebras. We obtain the $\Gamma$-graded $\mathcal{O}P$-interior algebra homomorphism
\[(\Ind_{\sigma}(\End_{\mathcal{O}}(\ddot{N})\otimes_{\mathcal{O}}\Res_{\sigma'}((A')_{\gamma'})))_{\hat{\gamma}}\to (\Ind_{\rho}(\End_{\mathcal{O}}(\ddot{N})\otimes_{\mathcal{O}}\Res_{\sigma'}((A')_{\gamma'})))_{\hat{\gamma}},\] which is an isomorphism by  the second part of isomorphism (\ref{5}). Now the homomorphism (\ref{4}) restricts to the $\Gamma$-graded $\mathcal{O}P$-interior algebra homomorphism
\[A_{\gamma}\to (\Ind_{\sigma}(\End_{\mathcal{O}}(\ddot{N})\otimes_{\mathcal{O}}\Res_{\sigma'}((A')_{\gamma'})))_{\hat{\gamma}},\] and  (\ref{5}) shows that this is in fact an isomorphism.

Conversely, let us first denote
\[\hat{B}:=(\Ind_{\sigma}(\End_{\mathcal{O}}(\ddot{N})\otimes_{\mathcal{O}}\Res_{\sigma'}((A')_{\gamma'})))_{\hat{\gamma}},\]
and we have seen that $\hat{B}$ is a $\Gamma$-graded $\mathcal{O}P$-interior algebra. Further, since $G_{\{b\}}$ is projective relative to $P_{\gamma}$ on $A_1$, we still have $b\in \Tr_P^G(A^P\gamma A^P).$ It follows, according to \cite[Proposition 3.6]{Puig2}, that there exists a $\Gamma$-graded $\mathcal{O}G$-interior algebra embedding
\[h:A\to \Ind_P^G(A_{\gamma}),\] sending $a\in A$ to $$\Tr_P^G(1\otimes i\otimes 1)\cdot (1\otimes iai\otimes 1)\cdot \Tr_P^G(1\otimes i\otimes 1),$$ where $i\in \gamma.$

By our assumptions we obtain the $\Gamma$-graded $\mathcal{O}G$-interior algebra isomorphism
\[g:\Ind_P^G(A_{\gamma})\to \hat{B},\] and then the obvious $\Gamma$-graded $\mathcal{O}G$-interior algebra embedding
\[\hat{B}\to \Ind_P^G(\Ind_{\sigma}(\End_{\mathcal{O}}(\ddot{N})\otimes_{\mathcal{O}}\Res_{\rho'}(A')))\simeq \hat{A}.\]
Clearly, we have the $\Gamma$-graded $\mathcal{O}G$-interior algebra embedding
\[\phi:A\to \hat{A},\] given by the composition of the above maps.  Then $\phi_1:A_1\to \hat{A}_1$ is an embedding of $G$-algebras.
It is also clear that $\phi_1(\gamma)=\hat{\gamma}$, and then $b$ determines a point $\hat{\beta}$ of $G$ on $\hat{A}_1,$ which is already a point of $G$ on $\hat{B}_1,$ with defect group $P_{\hat{\gamma}}$, so that we have the isomorphism
\begin{equation}
\label{6}\hat{\phi}(P_{\gamma}):\hat k_*\hat{\bar{N}}_{G}(P_{\gamma})\simeq \hat k_*\hat{\bar{N}}_{G}(P_{\hat{\gamma}}) \end{equation}
of ${\bar{N}}_{G}(P_{\gamma})$-graded algebras, and the isomorphisms
\begin{align*} V_{A_1}(P_{\gamma})  & \simeq \Res_{\hat{\phi}(P_{\gamma})}(V_{(\hat{A}_1)_{\hat{\beta}}}(P_{\hat{\gamma}})) \\
  & \simeq \Res_{\hat{\phi}(P_{\gamma})}(V_{(\hat{B}_1)_{\hat{\beta}}}(P_{\hat{\gamma}})).\end{align*}
Let $\hat{\alpha}$ be the point of $G$ on $\hat{A}_1$ such that $b\cdot \hat{\alpha}=\hat{\alpha}$, obtained via the $G$-algebra embedding
\begin{equation}
\label{7}\End_{\mathcal{O}}(M)^{1\times N'}\to \hat{A}_1.\end{equation}
 Using this and our assumption, we have the $\hat k_*{\hat{\bar{N}}}_G(P_{\gamma})$-module isomorphisms
\begin{align*} V_{A_1}(P_{\gamma}) & \simeq \Res_{\hat{\phi}(P_{\gamma})}(V_{M}((P\times N')_{\hat{\gamma}})) \\
          & \simeq \Res_{\hat{\phi}(P_{\gamma})}(V_{(\hat{A}_1)_{\hat{\alpha}}}(P_{\hat{\gamma}})).\end{align*}
Hence, by (\ref{6}) and (\ref{7}) we can identify $\hat{\alpha}$ with $\hat{\beta}$ and the structural morphisms $f:A_1\to (\hat{A}_1)_{\hat{\alpha}}$ with the restriction $A_1\to (\hat{B}_1)_{\hat{\beta}},$ of $\phi_1$.

Since $g_1\circ h_1$ is an embedding, there is a homomorphism of $(\mathcal{O}N,\mathcal{O}N)$-bimodules
\[r:(\hat{B}_1)_{\hat{\beta}}:=j(\hat{B}_1)j\to A_1,\]
which is also  $\mathcal{O}G$-linear, such that $r\circ g_1\circ h_1=\mathrm{id}_{A_1},$ where $g_1(h_1(b))=j.$ Now, for any $x\in N$, we get
\[xb=x(r (g_1(h_1(b)))=xr(j)=r(f(xb)),\] hence $f$ is an isomorphism.

Finally, we obtain the $G$-algebra isomorphism
\[A_1\simeq \End_{\mathcal{O}}(M)^{1\times N'}.\]
Now, by using \cite[Theorem 6.5]{Puig},  the fact that $M$ extends to $\ddot G$, and \cite[Theorem 5.1.2]{M}, the required statement follows.
\end{proof}

\begin{subsec}\label{sufficiency_for_morita_eq} Similarly to \cite[6.12]{Puig}, we show how by applying Theorem \ref{main}, we obtain all the possible choices of an indecomposable $\mathcal{O}\ddot{G}$-module inducing a Morita equivalence between $A_1$ and $A'_1$, and whose induction to $G\times G'$ induces a $\Gamma$-graded Morita equivalence between $A$ and $A'.$

Let $b$ and $b'$ be as in \ref{assumption_on_M}, and assume that $Q\leq G$ is a defect group of $b$ in $G$, while $Q'$ is a defect group of $b'$ in $G'$, such that $\ddot{Q}\leq \ddot G,$ where $\ddot{Q}$ runs through all the subgroups of $Q\times Q'$ such that the projections $\tau$ and $\tau'$ satisfy $\tau(\ddot{Q})=Q$ and $\tau'(\ddot{Q})=Q'.$

We consider the defect pointed group $Q_{\delta}$ of $b$ and the defect pointed group $Q'_{\delta'}$ of $b'$, and we look for all indecomposable $\mathcal{O}\ddot{Q}$-modules $\ddot{L}$ having vertex $\ddot{Q},$ such that the restrictions to $\Ker(\tau)$ and to $\Ker(\tau')$ are projective,  and there is a $\Gamma$-graded $\mathcal{O}Q$-interior algebra embedding
\[A_{\delta}\to \Ind_{\tau}(\End_{\mathcal{O}}(\ddot{L})\otimes_{\mathcal{O}}(\Res_{\tau'}(A'_{\delta'}))).\]

With these assumptions, denoting by $\xi'$ the restriction to $\ddot{Q}$ of the projection $G\times G'\to G'$,  note that there is the composition
\begin{align*}
f:A_{\delta}&\to \Ind_{\tau}(\End_{\mathcal{O}}(\ddot{L})\otimes_{\mathcal{O}}(\Res_{\tau'}(A'_{\delta'})))\\
           &\to \Ind_{\tau}(\End_{\mathcal{O}}(\ddot{L})\otimes_{\mathcal{O}}(\Res_{\xi'}(\mathcal{O}G')))\\
           &\to \Ind_{\ddot{Q}}^{Q\times G'}(\End_{\mathcal{O}}(\ddot{L}))^{1\times G'} \\
           &\to \End_{\mathcal{O}}(\Ind_{\ddot{Q}}^{G\times G'}(\ddot{L}))^{1\times G'}
\end{align*}
of  $\Gamma$-graded $\mathcal{O}Q$-interior algebra embeddings, that restricts to the embedding
\[f_1:(A_1)_{\delta}\to \End_{\mathcal{O}}(\Ind_{\ddot{Q}}^{\ddot G}(\ddot{L}))^{1\times N'}\]
of $Q$-algebras between the identity components.

Now $\hat{\delta}:=f_1(\delta)$ is a local point of $Q$ on $\End_{\mathcal{O}}(\Ind_{\ddot{Q}}^{\ddot G}(\ddot{L}))^{1\times N'},$ and at the same time, a point of $Q\times N'$ on $\End_{\mathcal{O}}(\Ind_{\ddot{Q}}^{\ddot G}(\ddot{L}))$, or equivalently, a point of $Q\times G'$ on $\End_{\mathcal{O}}(\Ind_{\ddot{Q}}^{G\times G'}(\ddot{L}))_1.$ Let $\bar{\delta}$ be a local point of $Q$ on $A_{\delta}.$ According to \cite[6.13]{Puig}, $\bar{\delta}$ determines a unique local pointed group $Q_{\ddot{\hat{\delta}}}$ on $\End_{\mathcal{O}}(\Ind_{\ddot{Q}}^{G\times G'}(\ddot{L}))^{1\times G'}.$ The point $\ddot{\hat{\delta}}$ is actually a point of $Q\times G'$ on  $(\End_{\mathcal{O}}(\Ind_{\ddot{Q}}^{G\times G'}(\ddot{L})))_{\hat{\delta}}$ satisfying
\[\ddot{Q}_{\ddot{L}}\leq (Q\times G')_{\ddot{\hat{\delta}}}.\]

Let $\ddot{\hat{\beta}}$ be a point of $G\times G'$ on $\End_{\mathcal{O}}(\Ind_{\ddot{Q}}^{G\times G'}(\ddot{L}))$ such that
\[(Q\times G')_{\ddot{\hat{\delta}}}\leq (G\times G')_{\ddot{\hat{\beta}}}.\]
There is a unique indecomposable $\mathcal{O}\ddot G$-summand $Y$  of $\Ind_{\ddot{Q}}^{\ddot G}(\ddot{L})$ such that $\ddot{\hat{\beta}}$ corresponds to  a unique isomorphisms class of indecomposable $\mathcal{O}(G\times G')$-modules determined by a direct summand of $\Ind_{\ddot G}^{G\times G'}(Y).$ Explicitly, we still have the inclusion
\[(Q\times G')_{\ddot{\hat{\delta}}}\leq (G\times G')_{\ddot{\hat{\beta}}}\]
on the $\Gamma$-graded $\mathcal{O}(G\times G')$-interior algebra  $\End_{\mathcal{O}}(\Ind_{\ddot G}^{G\times G'}(Y)).$ Hence $\ddot{\hat{\delta}}$ is a point of $Q\times G'$ on  $(\End_{\mathcal{O}}(\Ind_{\ddot G}^{G\times G'}(Y)))_{\hat{\delta}}.$

According to Lemma \ref{point_relation}, $\ddot{\hat{\beta}}$ determines a   point $\hat{\beta}$ of $G\times G'$ on the identity component $\End_{\mathcal{O}}(\Ind_{\ddot G}^{G\times G'}(Y))_1$ of $\End_{\mathcal{O}}(\Ind_{\ddot G}^{G\times G'}(Y)).$ We clearly have the inclusion
\[(Q\times G')_{\hat{\delta}}\leq (G\times G')_{\hat{\beta}}\]
of pointed groups on $\End_{\mathcal{O}}(\Ind_{\ddot G}^{G\times G'}(Y))_1.$ The $\mathcal{O}\ddot{Q}$-interior algebra embedding
\[\End_{\mathcal{O}}(\ddot{L})\to \End_{\mathcal{O}}(\Ind_{\ddot{Q}}^{\ddot G}(\ddot{L}))\]
shows that $\ddot{L}$ lies in the restriction to $\ddot{Q}$ of a unique $\mathcal{O}{\ddot G}$-indecomposable direct summand of $\Ind_{\ddot{Q}}^{\ddot G}(\ddot{L}).$ Taking into account that
\[\ddot{Q}_{\ddot{L}}\leq (G\times G')_{\ddot{\hat{\beta}}},\]
we deduce that $\ddot{L}\mid \Res_{\ddot{Q}}^{\ddot G}(Y),$ and then $\ddot{Q}_{\ddot{L}}$ is still a local pointed group on  $\End_{\mathcal{O}}(\Ind_{\ddot G}^{G\times G'}(Y)),$ hence on $\End_{\mathcal{O}}(\Ind_{\ddot G}^{G\times G'}(Y))_1,$ where we have the inclusions
\[\ddot{Q}_{\ddot{L}}\leq(Q\times G')_{\hat{\delta}}\leq (G\times G')_{\hat{\beta}}.\]

Now, the natural embedding of $\mathcal{O}\ddot G$-interior algebras
\[\End_{\mathcal{O}}(Y)\to \End_{\mathcal{O}}(\Ind_{\ddot{Q}}^{\ddot G}(\ddot{L}))\]
and \cite[2.11.3]{Puig}, since $\ddot{Q}_{\ddot{L}}\leq {\ddot G}_Y,$  prove that $\ddot{Q}_{\ddot{L}}$ is a defect pointed group of ${\ddot G}_Y,$ hence $\ddot{Q}$ is a vertex of $Y.$ Let $\ddot{R}_{\hat{\epsilon}}$ be a local pointed group on $\End_{\mathcal{O}}(\Ind_{\ddot{Q}}^{G\times G'}(\ddot{L}))_1$ such that
\[\ddot{Q}_{\ddot{L}}\leq \ddot{R}_{\hat{\epsilon}}\leq (Q\times G')_{\hat{\delta}}.\] Then $\ddot{R}_{\hat{\epsilon}}$ determines a local pointed group $\ddot{R}_{\ddot{\hat{\epsilon}}}$ on $\End_{\mathcal{O}}(\Ind_{\ddot{Q}}^{G\times G'}(\ddot{L})),$ which, according to \cite[2.11.3]{Puig}, is included in $\ddot{Q}_{\ddot{L}}.$ This shows that $\ddot{Q}_{\ddot{L}}$ is a defect pointed group for both $(Q\times G')_{\hat{\delta}}$ and $(G\times G')_{\hat{\beta}}.$

We have obtained the embedding
\[f_1:(A_1)^Q_{\delta}\to ((\End_{\mathcal{O}}(\Ind_{\ddot G}^{G\times G'}(Y)))_1)_{\hat{\delta}}^{Q\times G'}=\End_{\mathcal{O}}(Y)_{\hat{\delta}}^{Q\times N'},\]
and then
\begin{align*} N_G(Q_{\delta})/Q &\simeq N_{G\times G'}((Q\times G')_{\hat{\delta}})/Q\times G'  \\
                                 &\simeq N_{G\times G'}((Q\times N')_{\hat{\delta}})/Q\times G'.
\end{align*}
The embedding $f_1$ induces the  isomorphisms
\begin{align*}\hat{f_1}(Q_{\delta}):\hat k_*\hat{\bar{N}}_{G}(Q_{\delta})&\simeq \hat k_*\hat{\bar{N}}_{G\times G'}((Q\times G')_{\hat k_*{\delta}}) \\
     &\simeq \hat k_*\hat{\bar{N}}_{G\times G'}((Q\times N')_{\hat{\delta}})/Q\times G',
\end{align*}
of ${\bar{N}}_{G}(Q_{\delta})$-graded algebras, and then we have
\[V_{A_1}(Q_{\delta})\simeq \Res_{\hat{f_1}(Q_{\delta})}(V_Y((Q\times N')_{\hat{\delta}})).\]
\end{subsec}

Finally,  notice that the indecomposable $\mathcal{O}\ddot G$-summand $Y$  of $\Ind_{\ddot{Q}}^{\ddot G}(\ddot{L})$ introduced here is projective when restricted to $\mathcal{O}(N\times 1)$ and to $\mathcal{O}(1\times N').$ As a conclusion of the above discussion, by Theorem \ref{main} we obtain:

\begin{corollary} The $\Gamma$-graded $\mathcal{O}(G\times G')$-module $\ddot{Y}:=\Ind_{\ddot G}^{G\times G'}(Y)$ induces a $\Gamma$-graded Morita equivalence between $b\mathcal{O}G$ and $b'\mathcal{O}G'.$
\end{corollary}

\section{Graded basic Morita equivalences} \label{sect:basic}

By using   the  constructions of the previous section, we get the following graded version of  \cite[Corollary 7.4 ]{Puig}, which allows us to define the notion of basic graded Morita equivalence. The notations are those of Section \ref{sect:prelim} and of Theorem \ref{main}.

\begin{proposition} \label{equiv_statements_for_basic} Assume that $\ddot{M}$ defines a graded Morita equivalence between $A$ and $A'.$  Then the following assertions are equivalent:
\begin{enumerate}
\item[$1)$] $\sigma$ is a group isomorphism;
\item[$2)$] $\sigma'$ is a group isomorphism;
\item[$3)$] $p$ does not divide the rank of $\ddot{N}$ over $\mathcal{O};$
\item[$4)$] $S$ is a Dade $\ddot{P}$-algebra.
\end{enumerate}
\end{proposition}

\begin{proof} The implication $4)\Rightarrow 3)$ follows exactly as in the proof of \cite[Corollary 7.4]{Puig}. Further, we assume that $3)$ holds. Then, since $\Ker(\sigma)\leq 1\times N',$ $\Ker(\sigma')\leq N\times 1$ and $\ddot{N}\mid \Res^{\ddot G}_{\ddot{P}}(M),$ the assumptions made in \ref{s:ddotM} show that $\ddot{N}$ is a projective $\mathcal{O}\Ker(\sigma)$-module and  a projective $\mathcal{O}\Ker(\sigma')$-module. This fact forces $\sigma$ and $\sigma'$ to be injective group homomorphisms.  Now Proposition \ref{points_on_1_cmp} and the proof of Theorem \ref{main} show that $\hat{A}_{\hat{\gamma}}$ is a direct $\mathcal{O}$-summand of $\hat{A}_{\alpha},$ since $\gamma$ corresponds to $\hat{\gamma},$ $b\cdot \alpha=\alpha,$ $P_{\gamma}$ is a defect pointed group of $G_{\{b\}}$, and since the isomorphism $A_1\simeq (\hat{A}_1)_{\alpha}$ forces $A\simeq \hat{A}_{\alpha}.$  The argument  used in \cite[Remark 6.11]{Puig} guarantees that $\hat{A}_{\hat{\gamma}}$ has a $P$-stable $\mathcal{O}$-basis. The $\mathcal{O}P$-interior algebra embedding
\[\hat{A}_{\hat{\gamma}}\simeq (\Ind_{\sigma}(S\otimes_{\mathcal{O}}\Res_{\sigma'}(A'_{\gamma'})))_{\hat{\gamma}},\] given by Proposition \ref{points_on_1_cmp}, determines the $\mathcal{O}P$-interior algebra embedding
\[\hat{A}_{\hat{\gamma}}\to \Res_{\sigma^{-1}}(S)\otimes_{\mathcal{O}}\Res_{\sigma'\circ \sigma^{-1}}(A'_{\gamma'}).\] Finally, we apply \cite[Theorem 7.2]{Puig} to this last embedding.
\end{proof}

\begin{definition} \label{basic_graded_def} The $\Gamma$-graded module $\ddot{M}$ determines a {\it basic  graded Morita equivalence} between $A$ and $A'$ if $\ddot{M}$ determines a graded Morita equivalence between $A$ and $A'$, and any of the equivalent statements in Corollary \ref{equiv_statements_for_basic} hold.
\end{definition}

\begin{corollary} \label{c:truncation} Assume that the $\Gamma$-graded $\mathcal{O}(G\times G')$-module $\ddot{M}$ determines a basic graded Morita equivalence between $A$ and $A'.$ Let $\Lambda$ be a subgroup of $\Gamma$, and let $H:=\omega^{-1}(\Lambda)$ and  $H':={\omega'}^{-1}(\Lambda)$.

Then the $\Lambda$-graded $\mathcal{O}(H\times H')$-module $\ddot{M}_\Lambda=\bigoplus_{x\in \Lambda}\ddot{M}_x$ determines a basic graded Morita equivalence between $A_\Lambda$ and $A'_\Lambda.$
\end{corollary}

\begin{proof} We  have the isomorphism
\[A_\Lambda\simeq \End_{\mathcal{O}}(\ddot{M}_\Lambda)^{1\times H'}\] of $\Lambda$-graded algebras,
which means that $M$ induces an $\Lambda$-graded Morita equivalence between $A_\Lambda$ and $A'_\Lambda.$ Denote
\[\ddot H:=\{(g,g')\in H\times H'\mid \omega(g)=\omega(g') \}.\]
As $\mathcal{O}(H\times H')$-modules, we have that
\[M\mid \sum_{(h,h')\in [\ddot{H}\setminus \ddot{G}/\ddot{P}]}\Ind_{\ddot{P}^{(h,h')}\cap\ddot{H}}^{\ddot H}(\ddot{N}^{(h,h')}),\]
where $\ddot{P}\leq \ddot{G}$ is the above vertex of $M$ in $\ddot{G}$.  We may choose a vertex $\ddot{Q}$  of $M$ in $\ddot{H}$ such that $\ddot{Q}\leq \ddot{P}.$ Denote by $\sigma_{\ddot{Q}}$ and $\sigma'_{\ddot{Q}}$ the projections $\ddot{Q}\to Q$ and $\ddot{Q}\to Q'$ respectively, determined by the projections $\ddot{H}\to H$ and $\ddot{H}\to H'.$ We obtain the commutative diagram
\[\begin{xy} \xymatrix{ \ddot{P} \ar[r]^{\sigma} &P
\\  \ddot{Q}\ar[u]\ar[r]^{\sigma_{\ddot{Q}}} &Q\ar[u]    } \end{xy},\]
where the vertical maps are the inclusions. A similar commutative diagram exists for $P'$ and $Q'.$ Finally, if $\sigma$ is an isomorphism, $\sigma_{\ddot{Q}}$ is also an isomorphism.
\end{proof}

\section{The graded structure of the extended Brauer quotient}

\begin{subsec}  We keep the notations introduced in \ref{sect:prelim}. Let $Q$ be a $p$-subgroup of $G$ and consider the  subgroup
\[K:=\Aut^\Gamma(Q)=\{\varphi\in \Aut(Q)\mid \varphi(u)\in uN \mbox{ for all } u\in Q\}\]
of grade-preserving automorphisms of $Q$. For any $\varphi\in K$, consider the $\varphi$-twisted diagonal map
\[\Delta_\varphi:Q\to Q\times Q, \qquad u\mapsto(\varphi(u),u).\]
Denote by $N^K_G(Q)$ the inverse image of $K$ via the group homomorphism \[N_G(Q)\to \Aut(Q).\] It is clear that $C_N(Q)$ is a normal subgroup of $N^K_G(Q).$ Let $N_{\mathcal{O}G}^\varphi(Q)$ be the set of $\Delta_\varphi(Q)$-fixed elements of $\mathcal{O}G$, and let
\[\bar N^\varphi_{kG}(Q)=(\mathcal{O}G)(\Delta_\varphi(Q)).\]
The aim of this section is to point out that the extended Brauer quotient
\[\bar{N}^K_{kG}(Q)=\bigoplus_{\varphi\in K}\bar{N}^{\varphi}_{kG}(Q)\] introduced in \cite{Puig3} (see also \cite{CT} for the generalization to $N$-interior $G$-algebras), and the isomorphism of \cite[Theorem 3.5]{Puig3} admit the following $N_G^K(Q)/C_N(Q)$-graded structure.
\end{subsec}

\begin{proposition} \label{gr_ext_quotient}  There exists an isomorphism
\[\psi^K:kN_G^K(Q)\simeq \bar{N}^K_{kG}(Q)\]
of $N_G^K(Q)/C_N(Q)$-graded $N_G^K(Q)$-interior $N_G(Q)$-algebras.
\end{proposition}

\begin{proof} The set $G$ is a $Q\times Q$-invariant $\mathcal{O}$-basis of $\mathcal{O}G$, so it is also $\Delta_{\varphi}(Q)$-invariant for any $\varphi\in K.$ Hence for any $\varphi\in K,$ the subalgebra $\mathcal{O}G^{\Delta_{\varphi}(Q)}$ of $\Delta_{\varphi}(Q)$-fixed elements of $\mathcal{O}G$   has as basis the class sums \[\sum_{u\in Q}\varphi (u)gu^{-1},\] where $g\in G.$ Now, if $\varphi (u)gu^{-1}=g$, we get $\varphi (u)=u^{g^{-1}}$ for any $u\in Q,$ forcing $g\in N_G^K(Q).$

For any $x\in N_G^K(Q)$ and any $u\in Q$,  denote $\varphi_x(u)=u^{x^{-1}}$.  Then it is easy to see that we get
\[(\mathcal{O}Nz)^{\Delta_{\varphi_x}(Q)}\subseteq \mathcal{O}Nz,\]
as $\mathcal{O}$-modules, for any $z,x\in N_G^K(Q).$
We can organize the algebra
\[\mathcal{O}N_G^K(Q)N=\bigoplus_{z\in [N_G^K(Q)/C_N(Q)] }\mathcal{O}Nz,\] such that
any class $Nz$ is a $\Delta_{\varphi_x}(Q)$-invariant $\mathcal{O}$-basis  of the module $\mathcal{O}Nz$, which has a $\Delta_{\varphi_x}(Q)$-fixed element (for a unique representative $x\in [N_G^K(Q)/C_G(Q)]$) if and only if $zC_G(Q)=xC_G(Q).$

Since in our situation the map sending $x$ to $\varphi_x$ induces the monomorphism
\[N_{G}^K(Q)/C_G(Q)\simeq K,\] and since for any $z\in [N_G^K(Q)/C_N(Q)]$ there is $t\in [C_G(Q)/C_N(Q)]$ and $x\in [N_{G}^K(Q)/C_G(Q)]$ such that $z=xt,$ from the above and \cite[Proposition 2.5]{CT} we obtain that
\begin{align*}\bar{N}^K_{kG}(Q)&=\bar{N}_{kN}^{K}(Q)\bigoplus\left(\bigoplus_{\substack{z\in [N_G^K(Q)/C_N(Q)]\\z\notin N}  }\bar{N}_{\mathcal{O}Nz}^{\varphi_x}(Q)\right)\\
&=kN_N(Q)\bigoplus\left(\bigoplus_{\substack{z\in [N_G^K(Q)/C_N(Q)]\\z\notin N}  }\bar{N}_{\mathcal{O}Nz}^{\varphi_x}(Q)\right).
\end{align*}
It is also easy to check that \[kC_N(Q)z\subseteq N^{\varphi_x}_{\mathcal{O}Nz}(Q)\] for any $z\in [N_G^K(Q)/C_N(Q)].$ We define $\psi^K$ to be the sum
\[\psi^K=\bigoplus_{z\in [N_G^K(Q)/C_N(Q)]}\psi^K_z\]
of the family $\psi^K_z$ of module homomorphisms,  where each  $\psi^K_z$ is the restriction to $kC_N(Q)z$ of $\Br_{\Delta_{\varphi_x}(Q)}^{kG}.$ Since $\psi^K_1$ is an isomorphism, so is $\psi^K.$

We can define an $N_G(Q)$-action on $K$ as follows. For any $\varphi\in K$ and any $y\in N_G(Q),$ $\varphi^y$ is the automorphism of $Q$ given by
\[\varphi^y(u)=y^{-1}\varphi(yuy^{-1})y,\] for any $u\in Q.$ So, for any $\varphi\in K$ and any $y\in N_G(Q),$ the element $a\in \mathcal{O}G^{\Delta_{\varphi}(Q)}$ verifies $a^y\in \mathcal{O}G^{\Delta_{\varphi^y}(Q)}.$ Hence $\bar{N}^K_{kG}(Q)$ is a $N_G(Q)$-algebra. Now, by its definition,  $\psi^K$ is an isomorphism of $N_G(Q)$-algebras.
\end{proof}

\begin{remark}\label{ext_Br_Quot_for_local}  Let $T$ be a subgroup of $N^K_G(Q)$ such that $Q\leq T$,  and consider a pointed group $T_{\mu}$  on $\mathcal{O}N$ such that $\Br^{\mathcal{O}N}_Q(\mu)\neq 0.$ Then we also have that $\Br^{\mathcal{O}G}_Q(\mu)\neq 0.$ Choose a primitive idempotent $j\in \mu.$ The $\mathcal{O}Q$-interior algebra $j\mathcal{O}Gj$ is also $G/N$-graded, and consequently $\bar{N}^K_{i\mathcal{O}Gi}(Q)$ is $N_G^K(Q)/C_N(Q)$-graded as well. Moreover, the $G/N$-graded  $\mathcal{O}T$-interior algebra embedding
\[j\mathcal{O}Gj\to \mathcal{O}G\] gives the $N_G^K(Q)/C_N(Q)$-graded $\mathcal{O}T$-interior algebra embedding
\[\bar{N}^K_{jkGj}(Q)\to \bar{N}^K_{kG}(Q).\]
In $kC_G(Q)$, we clearly have the equality \[\Br_{Q}^{\mathcal{O}N}(j)=\Br_{Q}^{\mathcal{O}G}(j),\] and then \cite[Corollary 3.7]{Puig3} still holds in our case, giving the isomorphism
\[(kN_G^K(Q))_{\widehat{\Br_Q(\mu)}}\simeq \bar{N}^K_{(kG)_{\mu}}(Q)\] of $N_G^K(Q)/C_N(Q)$-graded $\mathcal{O}T$-interior algebras.
\end{remark}

\begin{remark} \label{ext_Br_Quot_for_S}  Let $P$ be a $p$-subgroup of $G$ and denote in this section, for the moment, the $kP$-interior algebra \[S= \End_k(W),\] where  $W$ be an endopermutation $kP$-module such that
\[(\End_k(W))(P)\neq 0.\] Let $Q$ be subgroup of $P,$  and let $x\in [N_G^K(Q)/C_G(Q)].$ Then  the $\mathcal{O}P$-interior algebra $S\otimes_{\mathcal{O}}\mathcal{O}G$ is $G/N$-graded and we have that
\begin{align*}(S\otimes_{\mathcal{O}}\mathcal{O}G)(\Delta_{\varphi_x}(Q))&=S(\Delta_{\varphi_x}(Q))\otimes_k (\mathcal{O}G)(\Delta_{\varphi_x}(Q))\\
&=\bigoplus_{t\in [C_G(Q)/C_N(Q), z=xt]}(S(\Delta_{\varphi_x}(Q))\otimes_k \bar{N}_{\mathcal{O}Nz}^{\varphi_x}(Q)).\end{align*}
Consequently, the $N_P^K(Q)$-interior algebra
\[\bar{N}^K_{S\otimes_{\mathcal{O}}\mathcal{O}G}(Q)=\bigoplus_{z\in [N_G^K(Q)/C_N(Q)]}S(\Delta_{\varphi_x}(Q))\otimes_k \bar{N}_{\mathcal{O}Nz}^{\varphi_x}(Q)\]
is also $N_G^K(Q)/C_N(Q)$-graded. We get yet another $N_G^K(Q)/C_N(Q)$-graded $N_P^K(Q)$-interior algebra by setting
\[S(Q)\otimes_k\bar{N}^K_{\mathcal{O}G}(Q)=\bigoplus _{z\in [N_G^K(Q)/C_N(Q)]}S(Q)\otimes_k \bar{N}_{\mathcal{O}Nz}^{\varphi_x}(Q).\]
\end{remark}

By using the notation introduced in Remarks \ref{ext_Br_Quot_for_local} and \ref{ext_Br_Quot_for_S}, and by adapting the proof of \cite[Proposition 3.9]{Puig3} we obtain:

\begin{proposition}\label{gr_on_tens_simple_alg} Let $Q_{\delta''}$ be a local pointed group on $A_1$, and let
\[\tilde{K}:=F_{S}(Q_{\delta''})\cap K.\]
Then there is an  isomorphism
\[\bar{N}_{S\otimes j(\mathcal{O}G)j}^K(Q)\simeq S(Q)\otimes_k\bar{N}^K_{j(\mathcal{O}G)j}(Q)\]
of $N_G^{\tilde{K}}(Q)/C_N(Q)$-graded $T$-interior algebras.
\end{proposition}

\section{Local basic graded Morita equivalences}

In this section we assume that the bimodule $\ddot{M}$ induces a basic graded Morita equivalence between $A$ and $A'$, as defined in Section \ref{sect:basic}.

\begin{subsec}  For any subgroup $Q$ in $P,$ denoting $\ddot{Q}=\sigma^{-1}(Q)$ and $\sigma'(\ddot{Q})=Q',$ we have that
\[N_P(Q)\simeq N_{\ddot{P}}(\ddot{Q})\simeq N_{P'}(Q'),\] and the  action of $N_{\ddot{G}}(\ddot{Q})$ on $\ddot{Q}$ determines the group isomorphisms
\[N_G(Q)/C_N(Q)\simeq N_{G'}(Q')/C_{N'}(Q') \]
and
\[N_G(Q)/C_G(Q)\simeq N_{G'}(Q')/C_{G'}(Q'). \]
\end{subsec}

\begin{subsec} \label{s:pointbij} We claim that there is a bijection between the set of local points of $Q$ on $(A_1)_{\gamma}$ and the set of local points of $Q'$ on $(A'_1)_{\gamma'}.$ Indeed, by our assumption, the  isomorphism \[e:A_{\gamma}\longrightarrow (\Ind_{\sigma}(\End_{\mathcal{O}}(\ddot{N})\otimes_{\mathcal{O}}\Res_{\sigma'}(A'_{\gamma'})))_{\hat{\gamma}}\] of $\Gamma$-graded $\mathcal{O}P$-interior algebras introduced in Theorem \ref{main} determines the embedding
\begin{equation}\label{8}   A_{\gamma}\to \Res_{\sigma^{-1}}(S)\otimes_{\mathcal{O}}\Res_{\sigma'\circ \sigma^{-1}}(A')_{\gamma'}\end{equation}
of $\Gamma$-graded $\mathcal{O}P$-interior algebra, and in particular, the embedding
\begin{equation}\label{9}   (A_1)_{\gamma}\to \Res_{\sigma^{-1}}(S)\otimes_{\mathcal{O}}\Res_{\sigma'\circ \sigma^{-1}}(A'_1)_{\gamma'}\end{equation}
of $\mathcal{O}(P\cap N)$-interior $P$-algebras.

Since $S(\ddot{Q})$ is a simple $k$-algebra, following \cite[7.6]{Puig}, we still have the  embedding
\begin{equation}\label{10}(A_1)_{\gamma}(Q)\to \Res_{\sigma^{-1}}(S(\ddot{Q}))\otimes_{\mathcal{O}}\Res_{\sigma'\circ \sigma^{-1}}(A_1')_{\gamma'}(Q').\end{equation}
of $N_P(Q)$-algebras, and this proves that claim.

Furthermore, if $Q_{\delta}$ and $Q'_{\delta'}$ correspond under this bijection, then by (\ref{9}) we also get the  embedding
\begin{equation}\label{11}A_{\delta}\to \Res_{\sigma^{-1}}(S_{\ddot{\delta}})\otimes_{\mathcal{O}}\Res_{\sigma'\circ \sigma^{-1}}A_{\delta'}\end{equation}
of $\Gamma$-graded $Q$-interior algebras, where $\ddot{\delta}$ is the unique point of $\ddot{Q}$ on $S$ lifting the identity of $S(\ddot{Q}).$
\end{subsec}

\begin{subsec} Consider the local pointed group $Q_{\delta}$ on $A_1$ such that $Q_{\delta}\leq P_{\gamma}$ corresponding to the local pointed group $Q'_{\delta'}$ on $A'_1$ such that $Q'_{\delta'}\leq P'_{\gamma'}.$ Let $b_{\delta}$ be the unique block  of $kC_N(Q)$ determined by $\delta$, and let $b'_{\delta'}$ be the unique block  of $kC_{N'}(Q')$ determined by $\delta'.$ If $N_G(Q)_{b_{\delta}}$ denotes the stabilizer in $N_G(Q)$ of the block $b_{\delta},$ and similarly $N_{G'}(Q')_{b'_{\delta'}}$ denotes the stabilizer of $b'_{\delta'},$ then we have the inclusions
\[N_G(Q_{\delta})\subseteq N_G(Q)_{b_{\delta}}\] and  \[N_{G'}(Q'_{\delta'})\subseteq N_{G'}(Q')_{b'_{\delta'}}.\]
\end{subsec}

We define $(G,\Gamma)$-fusions in terms of $\Gamma$-graded bimodules, in a way similar to \cite[7.2--7.5]{L2}

\begin{definition} We say that an element $\varphi\in K=\Aut^\Gamma(Q)$ is a $(G,\Gamma)$-fusion of $Q_\delta$ in $A_\gamma$ if there is an homogeneous isomorphism of some degree of $\Gamma$-graded bimodules
\[i\mathcal{O}Gj\simeq(i\mathcal{O}Gj)_\varphi,\]
where $i\in \gamma$ and $j\in\delta$.
\end{definition}

\begin{lemma}\label{the_tricky_lemma} With the above notations, and identifying $Q$, $Q'$ and $\ddot Q$ via $\sigma$ and $\sigma'$, the following statements hold.

{\rm1)} We have the group isomorphism
\[N^{{K}}_G(Q_{\delta})/C_N(Q)\simeq F_{A_\gamma}^{\Gamma}(Q_{\delta}).\]

{\rm2)}  The $\Gamma$-graded $\mathcal{O}Q$-interior algebra embedding {\rm(\ref{11})} gives the inclusion
\[F_{A_\gamma}^{\Gamma}(Q_{\delta})\subseteq F_S({\ddot Q}_{\ddot \delta})\]
and the equality
\[F_{A_\gamma}^{\Gamma}(Q_{\delta})=F_{{A'}_{\gamma'}}^{\Gamma} (Q'_{\delta'}).\]
\end{lemma}

\begin{proof} 1) As in \cite[Theorem 3.1]{Puig2} (see also \cite[7.2--7.5]{L2}), an homogeneous isomorphism between the $(i\mathcal{O}Gi,\mathcal{O}Q)$-bimodules $i\mathcal{O}Gj$ and $(i\mathcal{O}Gj)_\varphi$ is given by right multiplication with an element $x\in N_G^K(Q_\delta)$. The condition $x\in C_N(Q)$ means that $x$ defines a graded bimodule isomorphism of degree $1$ and a trivial automorphism of $Q$.

2) Both statements follow by the argument of \cite[1.17]{KP} and by taking into account that homogeneous isomorphisms of graded bimodules are induced by right multiplication with homogeneous invertible elements.
\end{proof}

\begin{lemma} Assume that $Q\subseteq N_G^K(Q_\delta)$. Let $T$ be a defect group of $b_{\delta}$ regarded as a primitive idempotent of $kC_N(Q)^{N^{K}_G(Q_{\delta})}.$ Then the group $T$ can be chosen such that $Q_\delta \le T_{\mu}\leq P_{\gamma},$ where $\nu $ is a local point of  $T$ on $kC_N(Q).$

Furthermore, in this situation we have that $T=N^{K}_P(Q_{\delta})$, and $T_{\Br_Q(\mu)}$  is a defect pointed group of $N^{K}_G(Q_{\delta})_{\{b_{\delta}\}}.$
\end{lemma}

\begin{proof} Let $\beta\in \mathcal{O}N^{N^{K}_G(Q_{\delta})}$ be the unique point lifting $b_{\delta}$ via the $\mathcal{O}$-algebra epimorphism
\[\Br_Q:\mathcal{O}N^{N^{K}_G(Q_{\delta})}\to kC_N(Q)^{N^{K}_G(Q_{\delta})}.\]
It is not difficult to see that this epimorphism actually restricts to the epimorphism
\[\Br_Q:\mathcal{O}N_T^{N^{K}_G(Q_{\delta})}\to kC_N(Q)_T^{N^{K}_G(Q_{\delta})},\]
and then $T$ is also a defect group of $\beta.$ The equality
\[\Br_Q(\delta)b_{\delta}=\Br_Q(\delta)\Br_Q(\beta)=\Br_Q(\delta)\] shows that we have $Q_{\delta}\leq N^{K}_G(Q_{\delta})_{\beta},$ and then we also get $N^{K}_G(Q_{\delta})_{\beta}\leq G_{\{b\}}.$
Now $b$ lies in
\[(\mathcal{O}N)^G_P\subseteq \sum_{x\in [N^{K}_G(Q_{\delta})\setminus G/P]}(\mathcal{O}N)^{N^{K}_G(Q_{\delta})}_{N^{K}_G(Q_{\delta})\cap P^x},\] and since $bj=j=jb$ for any $j\in \beta$ we obtain $T\leq P^x,$ for some $x\in G.$

Hence replacing each $P_{\gamma}$ and $Q_{\delta}$ by a $G$-conjugate, such that we still have $Q_{\delta}\leq P_{\gamma},$  we may assume that $T\leq P.$ More exactly $T\leq N_P(Q_{\delta}).$  The local pointed group $P_{\gamma}$  forces the existence of at least one local point $\mu \subseteq (\Od H)^{N_P(Q_{\delta})}$ with the property
\[N_P(Q_{\delta})_{\mu}\leq P_{\gamma}.\] The inclusion $Q\leq N_P(Q_{\delta})\leq N_G(Q_{\delta})$ shows that we may find, if necessary, some $G$-conjugate of $Q_{\delta }$ with
\[Q_{\delta}\leq N_P(Q_{\delta})_{\mu}\leq P_{\gamma}\]
and
\[Q_{\delta}\leq N_P(Q_{\delta})_{\mu}\leq N_G(Q_{\delta})_{\beta},\] since $\delta$ determines $\beta.$

Now $T,$ the defect group of $\beta,$ verifies $T\leq N_P(Q_{\delta}).$ The local point $\mu$ determines a local point $\bar{\mu}$ of $T$ on $\Od H$ with $T_{\bar{\mu}}\leq N_P(Q_{\delta})_{\mu}.$ The maximality of $T_{\bar{\mu}}$ forces the equality $T_{\bar{\mu}}=N_P(Q_{\delta})_{\mu}.$ The commutative diagram


\[\begin{xy} \xymatrix{ kC_N(Q)^T \ar[r]^{\Br^{kC_N(Q)}_T} &kC_N(T)
\\  \mathcal{O}N^T\ar[u]_{\Br^{\mathcal{O}N}_Q}\ar[ur]_{\Br_T^{\mathcal{O}N}}    } \end{xy}\]
proves $\Br^{kC_N(Q)}_T(\Br_Q(\mu))\neq 0$ and then $T_{\Br_Q(\mu)}$ is a defect pointed group of $N_{G}(Q_{\delta})_{b_{\delta}}.$
\end{proof}

\begin{subsec} Let $T'_{\mu'}\leq P'_{\gamma'}$ be the pointed group  corresponding to $T_{\mu}$ under the bijection given in \ref{s:pointbij}~(\ref{10}). It is not difficult to prove that $T'_{\Br_{Q'}(\mu')}$ is a defect pointed group of $N^{K}_{G'}(Q'_{\delta'})_{\{b'_{\delta'}\}}$ and that $T'=N^{K}_{P'}(Q'_{\delta'}).$ By Lemma \ref{the_tricky_lemma} we have
\[N^{K}_{G}(Q_{\delta })/C_N(Q)= N^{K}_{G'}(Q'_{\delta'})/C_{N'}(Q')\subseteq F_{S}({\ddot Q}_{\ddot{\delta}}).\]
Moreover,  embedding (\ref{8}) provides the   embedding
\[(\mathcal{O}G)_{\mu}\to S\otimes_{\mathcal{O}}(\mathcal{O}G')_{\mu'}\]
of $\Gamma$-graded $\mathcal{O}T$-interior algebras, which according to Proposition \ref{gr_ext_quotient} and Proposition \ref{gr_on_tens_simple_alg}, determines the   embedding
 \[(kN_G^{K}(Q_{\delta}))_{\widehat{\Br_Q(\mu)}}\to S(\ddot{Q})\otimes_k(kN_{G'}^{K}(Q'_{\delta'}))_{\widehat{\Br_{Q'}(\mu')}}\]
of $\mathcal{O}T$-interior $N^{K}_G(Q_{\delta})/C_N(Q)$-graded algebras. Recall also that here
\[S(\ddot{Q})=\End_{k}(\ddot{N}_{\ddot{Q}}),\] where ${\ddot N}_{\ddot{Q}}$ is a uniquely determined endo-permutation $kT$-module.
\end{subsec}

\begin{subsec} Consider the natural maps $\bar\omega:N^{K}_G(Q_{\delta})\to N^{K}_G(Q_{\delta})/C_N(Q)$ and $\bar\omega':N^{K}_{G'}(Q'_{\delta'})\to N^{K}_{G'}(Q'_{\delta'})/C_{N'}(Q')$, so in view of Lemma \ref{the_tricky_lemma} 2), we have, as in \ref{s:GN}, the corresponding subgroup
\[\{(g,g')\in N^{K}_G(Q_{\delta})\times  N^{K}_{G'}(Q'_{\delta'})\mid  \bar\omega (g)=\bar\omega(g') \} \]
of $N^{K}_G(Q_{\delta})\times  N^{K}_{G'}(Q'_{\delta'})$. It is clear that this subgroup coincides with $N_{\ddot{G}}^K(\ddot{Q}_{\delta\times \delta'})$, where $\delta\times\delta'$ is the unique point of $\ddot{Q}$ on $A_1\otimes A'_1$ such that $\Br_Q(\delta)\times \Br_{Q'}(\delta')\subseteq \Br_{\ddot{Q}}(\delta\times \delta')$.
\end{subsec}

By applying Theorem \ref{main} in this situation, we deduce the main result of this section.

\begin{theorem} \label{t:localequiv} Assume that $Q\subseteq N_G^K(Q_\delta)$. The indecomposable $\hat k_*\hat{\bar{N}}_{N^K_G(Q_{\delta})}(T_{\mu})$-module $V_{A_1}(T_{\mu})$ determines an indecomposable $kN_{\ddot{G}}^K(\ddot{Q}_{\delta\times \delta'})$-direct summand $Y$ of
\[\Ind_{\Delta(T)}^{N_{\ddot{G}}^K(\ddot{Q}_{\delta\times \delta'})}(\ddot{N}_{\ddot{Q}}),\] such that the bimodule
\[\Ind_{N_{\ddot{G}}^K(\ddot{Q}_{\delta\times \delta'})}^{N^{K}_G(Q_{\delta})\times   N^{K}_{G'}(Q'_{\delta'})}(Y)\] induces a basic $N^K_G(Q_\delta)/C_N(Q)$-graded Morita equivalence between $kN^{K}_G(Q_{\delta})b_{\delta}$ and  $kN^{K}_{G'}(Q'_{\delta'})b'_{\delta'}.$
\end{theorem}

\begin{remark} 1) If in Theorem \ref{t:localequiv} we have $G=N$, then we obtain the main result of Puig and Zhou \cite[Theorem 1.4]{Puig3}, and also of Hu \cite[Theorem 1.1]{Hu}.

2) On the other hand, if we assume that $b$ and $b'$ are the principal blocks of $\mathcal{O}N$ and $\mathcal{O}N'$ respectively, $\Gamma$ is a $p'$-group, $M$ is a splendid $\mathcal{O}(N\times N')$-module, and we take $K=1$ instead of $K=\Aut^\Gamma(Q)$ (hence $N^K_G(Q)=C_G(Q)$), the we obtain \cite[Corollary 3.9.a)]{Marcus}.
\end{remark}

\end{document}